\documentclass[11pt,reqno]{amsart}
\usepackage{amssymb, a4wide, amsmath, mathtools,xy}
\xyoption{all}
\usepackage{latexsym,pdfsync,xcolor,graphicx}
\usepackage{multirow}

\usepackage[T1]{fontenc}

\usepackage[utf8]{inputenc}

\usepackage{array}
\usepackage{upgreek}
\usepackage{pstricks-add}
\usepackage{enumerate}
\usepackage{orcidlink}
\usepackage[english]{babel}	
\usepackage{comment}
\allowdisplaybreaks

\usepackage{pgf,tikz}
\usepackage{wrapfig}
\usetikzlibrary{arrows}

\usepackage{float}
\usepackage{appendix}

\usepackage{hyperref}
\hypersetup{colorlinks,
    citecolor=black,
    filecolor=black,
    linkcolor=black,
    urlcolor=black}

\theoremstyle{plain}%
\newtheorem{theorem}{Theorem}[section]
\newtheorem{lemma}[theorem]{Lemma}%
\newtheorem{corollary}[theorem]{Corollary}%

\theoremstyle{definition}
\newtheorem{definition}[theorem]{Definition}%
\newtheorem{example}[theorem]{Example}%

\theoremstyle{remark}
\newtheorem{remark}[theorem]{Remark}

\title[Opposite brace triples, Hopf braces and matched pairs]{Opposite brace triples, Hopf braces and matched pairs of Hopf algebras}

\author[R. González Rodríguez]{Ramón González Rodríguez
\orcidlink{0000-0003-3061-6685}}
\author[B. Ramos Pérez]{Brais Ramos Pérez
\orcidlink{0009-0006-3912-4483}}

\address[R. González Rodríguez, B. Ramos Pérez]{CITMAga, 15782 Santiago de Compostela, Spain}
\address[R. González Rodríguez]{Departamento de Matemática Aplicada II, Universidade de Vigo, E.E. Telecomunicación, 36310 Vigo, Spain}
\email{rgon@dma.uvigo.es}
\address[B. Ramos Pérez]{Departamento de Matemáticas, Universidade de Santiago de Compostela, Facultade de Matemáticas, 15782 Santiago de Compostela, Spain}
\email{braisramos.perez@usc.es}



\begin{document}
	
\begin{abstract}
     In this paper the category of opposite brace triples is introduced in a general braided monoidal setting. Under cocommutativity, it is proved to be isomorphic to the category of Hopf braces. Furthermore, if one considers the subcategories arising from fixing one of the underlying Hopf algebras, then these two categories are also isomorphic to the category of matched pairs over that Hopf algebra.      
 \end{abstract}
	
\maketitle

\section{Introduction}
In \cite{AGV}, Iván Angiono, César Galindo, and Leandro Vendramin introduced Hopf braces as a Hopf-theoretic generalization of skew braces \cite{GV}. Informally, a Hopf brace $\mathbb{H}=(H_{1},H_{2})$ consists of two Hopf algebra structures on the same object, sharing a common coalgebra, and linked by a compatibility condition between their products. Although originally defined in the category of vector spaces, it is well known that this notion extends naturally to any braided monoidal category (cf. \cite{VGRRMod}). In a Hopf brace structure, note that the roles of $H_{1}$ and $H_{2}$ are not symmetric: $(H_{1},H_{2})$ being a Hopf brace does not entail that $(H_{2},H_{1})$ is also a Hopf brace. If such a situation occurs, then $\mathbb{H}$ is referred to as a bi-Hopf brace, extending the notion of bi-skew brace given by Childs in \cite{Childs}.

These structures play a key role in the construction of solutions to the Quantum Yang-Baxter Equation \cite{Bax,Yang} (for short, QYBE), thus establishing their relevance in both algebraic and physical contexts. More concretely, skew braces are fundamental in the construction and classification of non-degenerate and non-necessarily involutive set-theoretical solutions (see \cite{Bach}), whereas, as established in \cite[Corollaries 2.4 and 2.5]{AGV}, cocommutative Hopf braces also produce solutions to QYBE which are involutive in the case that the first Hopf algebra structure $H_{1}$ of the Hopf brace is commutative.

Given the significance of Hopf braces, it is desirable to develop new objects characterizing this category--at least in the cocommutative setting--thereby opening new avenues for the construction of solutions to QYBE; this is precisely the aim of this work. 

To address this objective, a prior example appearing in the literature was the category of brace triples introduced in \cite{VGRHac}, which is proved to be isomorphic to the category of Hopf braces under cocommutativity. It is worth noting that the category of Hopf braces has been shown to be equivalent to the categories of invertible 1-cocycles \cite{AGV,VGRRMod}, post-Hopf algebras, and relative Rota-Baxter operators \cite{LST}. In order to establish these equivalences and to define explicitly the respective functors to the category of Hopf braces, all these constructions share a common feature: given a Hopf algebra, say $H$, involved in a structure of invertible 1-cocycle, post-Hopf algebra or relative Rota-Baxter operator, it is deformed yielding a new Hopf algebra $H_{{\sf new}}$ such that the pair $(H,H_{{\sf new}})$ forms a Hopf brace. In all the cases, the product of $H_{{\sf new}}$ is of this form 
\begin{gather}\label{munew}\mu_{H}^{{\sf new}}=\mu_{H}\circ(H\otimes \gamma_{H})\circ(\delta_{H}\otimes H),\end{gather}
where the morphism $\gamma_{H}\colon H\otimes H\rightarrow H$ takes a specific form in each case. To give a particular example, for an invertible 1-cocycle $(\pi\colon A\rightarrow H,\phi_{H})$, the morphism $\gamma_{H}=\phi_{H}\circ (\pi^{-1} \otimes H)$, where $\phi_{H}\colon A\otimes H\rightarrow H$ is an action inducing a left $A$-module algebra structure on $H$. Therefore, a brace triple $(H,\gamma_{H},T_{H})$ captures the conditions required for an arbitrary morphism $\gamma_{H}$ in order the pair $(H,H_{{\sf new}})$ to be a Hopf brace, where $T_{H}\colon H\rightarrow H$ is an endomorphism that, in loose terms, acts as the antipode of $H_{{\sf new}}$.

In all the cases discussed in the previous paragraph, the deformed Hopf algebra $H_{{\sf new}}$ plays the role of $H_{2}$ in the Hopf brace structure. Therefore, since the roles of $H_{1}$ and $H_{2}$ are not symmetric in a Hopf brace, it is natural to consider the same problem as before but now with $H_{{\sf new}}$ playing the role of $H_{1}$. What conditions must $\gamma_{H}$ and $T_{H}$ satisfy in this situation? Determining these conditions is the main objective of this work, giving rise to the category of opposite brace triples. This category is shown to be isomorphic to the category of Hopf braces under cocommutativity (Theorem \ref{isomain}), this being the main outcome of this paper.   

It should be noted that, although less common than the previous case in existing work, there are also examples where the deformed Hopf algebra serves as $H_{1}$. This is precisely what happens when one obtains a Hopf brace from a matched pair of Hopf algebras, as shown in \cite[Section 3]{AGV}. In addition, the form of the product $\mu_{H}^{{\sf new}}$ in this case follows the same pattern as in \eqref{munew}. This correspondence gives a categorical isomorphism between Hopf braces and matched pairs of Hopf algebras \cite[Theorem 3.3]{AGV} recalled here in Theorem \ref{isoHBrMP}, which, combined with Theorem \ref{isomain}, also yields an isomorphism between the category of opposite brace triples and that of matched pairs of Hopf algebras (Corollary \ref{coroMP}).

\section{Preliminaries}
Throughout this paper, ${\sf C}$ denotes a strict braided monoidal category with tensor functor $\otimes$, unit object $K$, and braiding $c$. Recall that a strict monoidal category is a monoidal category in which the associativity, left and right unit constraints are identities, which makes the notation and arguments much easier to follow. This assumption does not entail any loss of generality because any monoidal category is monoidally equivalent to a strict one by the well-known Mac Lane's coherence theorem (see \cite{MacLane} for more details) and, consequently, any result proved in a strict setting extends automatically to the non-strict one with the appropriate modifications.

For the reader's convenience, in this section we recall some standard definitions that we will use recurrently in the development of the article.

\begin{definition}
An algebra in {\sf C} is a triple $A=(A,\eta_{A},\mu_{A})$ where $\eta_{A}\colon K\rightarrow A$ (unit) and $\mu_{A}\colon A\otimes A\rightarrow A$ (product) are morphisms in {\sf C} satisfying the following conditions:
\begin{align*}
\mu_{A}\circ (\eta_{A}\otimes A)={\rm id}_{A}=\mu_{A}\circ(A\otimes\eta_{A})\quad&\textnormal{(unit property)},\\\mu_{A}\circ(\mu_{A}\otimes A)=\mu_{A}\circ(A\otimes\mu_{A})\quad&\textnormal{(associativity)}.
\end{align*}
	
Note that, given $B=(B,\eta_{B},\mu_{B})$ another algebra in {\sf C}, the tensor product $A\otimes B$ admits an algebra structure whose unit and product are defined as follows:
$$
\eta_{A\otimes B}\coloneqq \eta_{A}\otimes\eta_{B}, \quad
\mu_{A\otimes B}\coloneqq (\mu_{A}\otimes\mu_{B})\circ(A\otimes c_{B,A}\otimes B).
$$
	
Moreover, a morphism $f\colon A\rightarrow B$ in ${\sf C}$ is said to be an algebra morphism if it preserves the unit and the product, that is,
$$
f\circ\eta_{A}=\eta_{B}, \quad f\circ\mu_{A}=\mu_{B}\circ(f\otimes f).
$$
	
Dually, we will say that a triple $C=(C,\varepsilon_{C},\delta_{C})$ is a coalgebra in {\sf C} if  $\varepsilon_{C}\colon C\rightarrow K$ (counit) and $\delta_{C}\colon C\rightarrow C\otimes C$ (coproduct) are morphisms in {\sf C} satisfying the following conditions:
\begin{align*}
(\varepsilon_{C}\otimes C)\circ\delta_{C}={\rm id}_{C}=(C\otimes\varepsilon_{C})\circ\delta_{C}\quad&\textnormal{(counit property)},\\(\delta_{C}\otimes C)\circ\delta_{C}=(C\otimes\delta_{C})\circ\delta_{C}\quad&\textnormal{(coassociativity)}.
\end{align*}
	
Likewise, if $D=(D,\varepsilon_{D},\delta_{D})$ is another coalgebra in {\sf C}, the tensor product $C\otimes D$ is a coalgebra with the following counit and coproduct:
$$
\varepsilon_{C\otimes D}\coloneqq\varepsilon_{C}\otimes\varepsilon_{D},\quad
\delta_{C\otimes D}\coloneqq(C\otimes c_{C,D}\otimes D)\circ(\delta_{C}\otimes\delta_{D}).
$$
	
Besides, a morphism $g\colon C\rightarrow D$ in ${\sf C}$ is a coalgebra morphism if the equalities
$$
\varepsilon_{D}\circ g=\varepsilon_{C}, \quad\delta_{D}\circ g=(g\otimes g)\circ\delta_{C}
$$
hold, i.e., $g$ preserves the counit and it is comultiplicative.
\end{definition}

\begin{definition}
Let $C$ be a coalgebra and let $A$ be an algebra in ${\sf C}$. Let $f,g:C\rightarrow A$ be a pair of morphisms in ${\sf C}$. We define the convolution product of $f$ and $g$ as 
\[f\ast g\coloneqq\mu_{A}\circ (f\otimes g)\circ\delta_{C}.\]

The set of morphisms in ${\sf C}$ from $C$ to $A$, denoted by $\operatorname{Hom}_{{\sf C}}(C,A)$, is a monoid with the convolution product whose unit is given by $ \eta_{A}\circ\varepsilon_{C}=\varepsilon_{C}\otimes\eta_{A}$.
\end{definition}

When an object in the category ${\sf C}$ is endowed with both an algebra and a coalgebra structure that are compatible with each other, we obtain the notion of a bialgebra.
\begin{definition}
	A 5-tuple $B=(B,\eta_{B},\mu_{B},\varepsilon_{B},\delta_{B})$ is said to be a bialgebra in {\sf C} if $(B,\eta_{B},\mu_{B})$ is an algebra in {\sf C}, $(B,\varepsilon_{B},\delta_{B})$ is a coalgebra in {\sf C} and $\eta_{B}$ and $\mu_{B}$ are coalgebra morphisms (equivalently, if $\varepsilon_{B}$ and $\delta_{B}$ are algebra morphisms).
	
	If $D=(D,\eta_{D},\mu_{D},\varepsilon_{D},\delta_{D})$ is another bialgebra in ${\sf C}$, a morphism in ${\sf C}$ $f\colon B\rightarrow D$ is a bialgebra morphism if it is simultaneously an algebra and a coalgebra morphism.
\end{definition}

In the following definition we recall the notion of Hopf algebra, which will be one of the key structures driving the development of the paper.
\begin{definition}
	A 6-tuple $H=(H,\eta_{H},\mu_{H},\varepsilon_{H},\delta_{H},\lambda_{H})$ is said to be a Hopf algebra in {\sf C} when $(H,\eta_{H},\mu_{H},\varepsilon_{H},\delta_{H})$ is a bialgebra such that there exists an endomorphism $\lambda_{H}\colon H\rightarrow H$, called the antipode, satisfying the following identity:
	\begin{gather}\label{antipode}
		\lambda_{H}\ast {\rm id}_{H}=\varepsilon_{H}\otimes\eta_{H}={\rm id}_{H}\ast \lambda_{H}.
	\end{gather}
\end{definition}

The equation \eqref{antipode} means that the antipode $\lambda_{H}$ is the inverse of the identity for the convolution product in $\operatorname{Hom}_{{\sf C}}(H,H)$ and, as a consequence, it is unique.

A Hopf algebra $H$ is said to be commutative when the underlying algebra is commutative, i.e., if the equality $\mu_{H}\circ c_{H,H}=\mu_{H}$, whereas $H$ is said to be cocommutative when the underlying coalgebra is cocommutative, that is, in case that $c_{H,H}\circ\delta_{H}=\delta_{H}$.

In any Hopf algebra $H$, its antipode satisfies the following properties: it is antimultiplicative and anticomultiplicative, which means that the equalities
\begin{gather}\label{a-antip1}
	\lambda_{H}\circ\mu_{H}=\mu_{H}\circ c_{H,H}\circ(\lambda_{H}\otimes\lambda_{H}),\\
	\label{a-antip2}
	\delta_{H}\circ\lambda_{H}=(\lambda_{H}\otimes\lambda_{H})\circ c_{H,H}\circ \delta_{H}
\end{gather}
hold, and $\lambda_{H}$ also preserves the unit and the counit, i.e.,
\begin{gather}\label{u-antip1}
	\lambda_{H}\circ\eta_{H}=\eta_{H},\\
	\label{u-antip2}\varepsilon_{H}\circ\lambda_{H}=\varepsilon_{H}.
\end{gather}

Thus, it is a direct consequence of the previous equalities that, when $H$ is commutative, $\lambda_{H}$ is an algebra morphism and, in case that $H$ is cocommutative, $\lambda_{H}$ is a coalgebra morphism. In addition, when $H$ is commutative or cocommutative, the equality 
\begin{gather}\label{lambdasquareid}
	\lambda_{H}\circ\lambda_{H}={\rm id}_{H}
\end{gather}
holds, i.e., $\lambda_{H}$ is an involution.

Given $B=(B,\eta_{B},\mu_{B},\varepsilon_{B},\delta_{B},\lambda_{B})$ another Hopf algebra in ${\sf C}$, a morphism $f\colon H\rightarrow B$ is a Hopf algebra morphism if it is a bialgebra morphism. Note that, in this case, $f$ commutes with the antipodes, that is,
\begin{gather}\label{morant}
	\lambda_{B}\circ f=f\circ\lambda_{H}.
\end{gather}

Note that, if $H$ is a Hopf algebra whose antipode $\lambda_{H}$ is an isomorphism in ${\sf C}$ (this situation occurs, for example, when $H$ is finite \cite[Corollary 1]{PuraE}), then $$H^{{\rm op}}=(H,\eta_{H},\mu_{H}^{{\rm op}},\varepsilon_{H},\delta_{H},\lambda_{H}^{-1}),$$
where $\mu_{H}^{{\rm op}}\coloneqq\mu_{H}\circ c_{H,H}^{-1}$, is a Hopf algebra in the mirror category of ${\sf C}$, $\overline{{\sf C}}$, which is the same as ${\sf C}$, but equipped with the braiding $c^{-1}$.

In \cite[Corollary 5]{Sch}, Schauenburg shows that, if $H$ is a cocommutative Hopf algebra in ${\sf C}$, the braiding $c_{H,H}$ is involutive, that is,
\begin{gather}\label{ccb}
	c_{H,H}\circ c_{H,H}={\rm id}_{H\otimes H}.
\end{gather} 
Therefore, in the cocommutative setting, $H^{{\rm op}}=(H,\eta_{H},\mu_{H}^{{\rm op}}=\mu_{H}\circ c_{H,H},\varepsilon_{H},\delta_{H},\lambda_{H})$ is also a Hopf algebra in ${\sf C}$ since $\lambda_{H}$ and $c_{H,H}$ are involutions under the cocommutativity assumption on  $H$ .

We now review the notion of a left module over a Hopf algebra; however, each of the definitions below admits a right-handed analogue, which can be formulated easily.
\begin{definition}
Let $H$ be a Hopf algebra in ${\sf C}$. A left $H$-module in ${\sf C}$ is a pair $(M,\varphi_{M})$ where $\varphi_{M}\colon H\otimes M\rightarrow M$ (action) is a morphism in ${\sf C}$ which satisfies the following equalities:
\begin{gather}\label{actioneta}
\varphi_{M}\circ(\eta_{H}\otimes M)={\rm id}_{M},\\
\label{actionprod}
\varphi_{M}\circ(H\otimes\varphi_{M})=\varphi_{M}\circ(\mu_{H}\otimes M).
\end{gather}

Let $(N,\varphi_{N})$ be a left $H$-module. A morphism $f\colon M\rightarrow N$ is a morphism of left $H$-modules (or, in other words, $f$ is $H$-linear) if 
\begin{gather*}\label{mod_mor}
	f\circ\varphi_{M}=\varphi_{N}\circ(H\otimes f).
\end{gather*}

With the previous morphisms, these objects constitute a category denoted by ${}_{H}{\sf Mod}$.
\end{definition}

\begin{definition}
Let $H$ be a Hopf algebra and $A$ an algebra in ${\sf C}$. A left $H$-module $(A,\varphi_{A})$ is said to be a left $H$-module algebra if $\eta_{A}$ and $\mu_{A}$ are $H$-linear morphisms, that is, if the following equalities
\begin{gather}
\label{Hmodalg1}\varphi_{A}\circ(H\otimes\eta_{A})=\varepsilon_{H}\otimes\eta_{A},\\		
\label{Hmodalg2}\varphi_{A}\circ(H\otimes\mu_{A})=\mu_{A}\circ\varphi_{A\otimes A},
\end{gather} 
hold, where $\varphi_{A\otimes A}\coloneqq(\varphi_{A}\otimes\varphi_{A})\circ(H\otimes c_{H,A}\otimes A)\circ(\delta_{H}\otimes A\otimes A)$.
\end{definition}
\begin{example}
Given a Hopf algebra $H$, $(H,\varphi_{H}^{{\rm ad}})$ is a left $H$-module algebra where $\varphi_{H}^{{\rm ad}}$ is the so-called adjoint action defined by 
\[\varphi_{H}^{{\rm ad}}\coloneqq \mu_{H}\circ(\mu_{H}\otimes\lambda_{H})\circ(H\otimes c_{H,H})\circ(\delta_{H}\otimes H).\]
\end{example}

\begin{definition}
Let $H$ be a Hopf algebra and $C$ a coalgebra in ${\sf C}$. A left $H$-module $(C,\varphi_{C})$ satisfying that $\varepsilon_{C}$ and $\delta_{C}$ are $H$-linear morphisms is said to be a left $H$-module coalgebra. In other words, $(C,\varphi_{C})$ is a left $H$-module coalgebra if the equalities 
\begin{equation}
\label{Hmodcoalg1}\varepsilon_{C}\circ\varphi_{C}=\varepsilon_{H}\otimes\varepsilon_{C},
\end{equation}
\begin{equation}\label{Hmodcoalg2}\delta_{C}\circ\varphi_{C}=\varphi_{C\otimes C}\circ(H\otimes\delta_{C})
\end{equation} 
hold. Note that \eqref{Hmodcoalg1} and \eqref{Hmodcoalg2} are equivalent to the fact that $\varphi_{C}$ is a coalgebra morphism.
\end{definition}

Let $H$ and $B$ be Hopf algebras in ${\sf C}$. If $(B,\varphi_{B})$ is both a left $H$-module algebra and a left $H$-module coalgebra, then it will be called a left $H$-module algebra-coalgebra.

The second family of structures we will work with in this paper are Hopf braces, a generalization of the usual Hopf algebras introduced by Angiono, Galindo and Vendramin \cite{AGV} in the vector space setting as the quantum version of a skew brace \cite{GV}. In a general braided monoidal framework, these objects are defined as follows.

\begin{definition}
	Let $(H,\varepsilon_{H},\delta_{H})$ be a coalgebra in {\sf C} and let us assume that $H$ admits two different algebra structures in {\sf C}: $(H,\eta_{H}^{1},\mu_{H}^{1})$ and $(H,\eta_{H}^{2},\mu_{H}^{2})$. We will say that a 9-tuple
	\[(H,\eta_{H}^{1},\mu_{H}^{1},\eta_{H}^{2},\mu_{H}^{2},\varepsilon_{H},\delta_{H},\lambda_{H}^{1},\lambda_{H}^{2})\]
	is a Hopf brace in {\sf C} if the following requirements hold:
	\begin{itemize}
		\item[(i)] $H_{1}=(H,\eta_{H}^{1},\mu_{H}^{1},\varepsilon_{H},\delta_{H},\lambda_{H}^{1})$ is a Hopf algebra in {\sf C}.
		\item[(ii)]  $H_{2}=(H,\eta_{H}^{2},\mu_{H}^{2},\varepsilon_{H},\delta_{H},\lambda_{H}^{2})$ is a Hopf algebra in {\sf C}.
		\item[(iii)] The following identity involving the products $\mu_{H}^{1}$ and $\mu_{H}^{2}$ holds:
		\begin{gather}\label{compatHbrace}
			\mu_{H}^{2}\circ(H\otimes\mu_{H}^{1})=\mu_{H}^{1}\circ(\mu_{H}^{2}\otimes\Gamma_{H_{1}})\circ(H\otimes c_{H,H}\otimes H)\circ(\delta_{H}\otimes H\otimes H),
		\end{gather}
		where 
		\begin{equation}\label{def_GammaH1}
			\Gamma_{H_{1}}\coloneqq \mu_{H}^{1}\circ(\lambda_{H}^{1}\otimes\mu_{H}^{2})\circ(\delta_{H}\otimes H).
		\end{equation}
	\end{itemize}
	
	Following the notation introduced in \cite{RGON}, we will denote Hopf braces by $\mathbb{H}=(H_{1},H_{2})$ or, when there is no confusion between the Hopf algebras involved, only by $\mathbb{H}$. A Hopf brace $\mathbb{H}$ is said to be cocommutative if $c_{H,H}\circ\delta_{H}=\delta_{H}$.
	
	These structures constitute a category whose morphisms are defined as follows: Given another Hopf brace $\mathbb{B}=(B_{1},B_{2})$, a morphism $f\colon \mathbb{H}\rightarrow \mathbb{B}$ is a morphism of Hopf braces if $f\colon H_{1}\rightarrow B_{1}$ and $f\colon H_{2}\rightarrow B_{2}$ are both Hopf algebra morphisms in ${\sf C}$. The category of Hopf braces will be denoted by ${\sf HBr}$. Cocommutative Hopf braces form a full subcategory of ${\sf HBr}$ which we will denote by ${\sf cocHBr}$.
\end{definition}

Given a Hopf algebra $H$ in ${\sf C}$, ${\sf HBr}(H)$ will denote the full subcategory of ${\sf HBr}$ whose objects are of the form $\mathbb{H}=(H_{1},H)$, that is, for any Hopf brace in ${\sf HBr}(H)$, the Hopf algebra $H_{2}$ coincides with the fixed Hopf algebra $H$. In the case that $H$ is cocommutative, we will denote this subcategory by ${\sf cocHBr}(H)$.

\begin{example}[Skew braces]\label{skewbraces}
	Since a Hopf algebra in the symmetric monoidal category of sets, {\sf Set}, is just a group, particularizing the previous notion of a Hopf brace with ${\sf C}={\sf Set}$ yields the notion of a skew brace, in which we have a pair of groups, $(G,\cdot)$ and $(G,\circ)$, satisfying the equation 
	\[a\circ(b\cdot c)=(a\circ b)\cdot a^{-1}\cdot(a\circ c), \textnormal{ for all }a,b,c\in G,\]
	which corresponds to \eqref{compatHbrace} in this setting. Therefore, any skew brace is an example of a Hopf brace considering that ${\sf C}={\sf Set}$.
	
	To refer to the the first group structure $(G,\cdot)$ involved in a skew brace, we will use the notation $G_{\cdot}$, while for the second group structure $(G,\circ)$, we will use $G_{\circ}$.
\end{example}
\begin{example}\label{galg}
Let $\mathbb{K}$ be a field. By linearization, it is well known that if $(G,\cdot,\circ)$ is a skew brace, then the respective group algebras of $G_{\cdot}$ and $G_{\circ}$, $\mathbb{K}[G_{\cdot}]$ and $\mathbb{K}[G_{\circ}]$, constitute a Hopf brace in ${\sf C}={}_{\mathbb{K}}{\sf Vect}$, denoted by $\mathbb{K}[\mathbb{G}]=(\mathbb{K}[G_{\cdot}],\mathbb{K}[G_{\circ}])$.
\end{example}
\begin{example}
	Any Hopf algebra $H=(H,\eta_{H},\mu_{H},\varepsilon_{H},\delta_{H},\lambda_{H})$ in ${\sf C}$ leads to a trivial Hopf brace in ${\sf C}$: $\mathbb{H}_{{\rm triv}}=(H,H)$. Indeed, when both Hopf algebras coincide in a Hopf brace structure, the compatibility condition \eqref{compatHbrace} always holds. In this case, $\Gamma_{H}=\varepsilon_{H}\otimes H$.
\end{example}

Given a Hopf brace $\mathbb{H}=(H_{1},H_{2})$, it is satisfied that $\eta_{H}^{1}=\eta_{H}^{2}$ \cite[Remark 1.3]{AGV}. Therefore, from now on we will denote both units by $\eta_{H}$. Besides, $(H_{1},\Gamma_{H_{1}})$ is a left $H_{2}$-module algebra \cite[Lemma 1.8]{AGV} and, reciprocally, this property characterizes the Hopf brace structure \cite[Theorem 1.16]{RGONRAMOS}. In the proof of the previous result, an important equality is the following:
\begin{gather}\label{agv1}
	\Gamma_{H_{1}}\circ(H\otimes\lambda_{H}^{1})=\mu_{H}^{1}\circ((\lambda_{H}^{1}\circ\mu_{H}^{2})\otimes H)\circ (H\otimes c_{H,H})\circ(\delta_{H}\otimes H),
\end{gather}
which will be useful in the development of this article (see \cite[Lemma 1.7]{AGV}).

Moreover, by \eqref{def_GammaH1}, coassociativity of $\delta_{H}$, associativity of $\mu_{H}^{1}$ and \eqref{antipode} for $H_{1}$, we obtain the following expression for $\mu_{H}^{2}$:
\begin{gather}\label{mu2-exp}
	\mu_{H}^{2}=\mu_{H}^{1}\circ(H\otimes \Gamma_{H_{1}})\circ(\delta_{H}\otimes H).
\end{gather}

In addition, by \eqref{mu2-exp} and the fact that $(H_{1},\Gamma_{H_{1}})$ is a left $H_{2}$-module whenever $\mathbb{H}$ is a Hopf brace, the following expression for $\mu_{H}^{1}$ is obtained:
\begin{gather}
	\label{mu1-exp}
	\mu_{H}^{1}=\mu_{H}^{2}\circ(H\otimes(\Gamma_{H_{1}}\circ(\lambda_{H}^{2}\otimes H)))\circ(\delta_{H}\otimes H).
\end{gather}
Indeed,
\begin{itemize}
	\itemindent=-32pt 
	\item[ ]$\hspace{0.38cm}\mu_{H}^{2}\circ(H\otimes(\Gamma_{H_{1}}\circ(\lambda_{H}^{2}\otimes H)))\circ(\delta_{H}\otimes H)$
	\item[]$=\mu_{H}^{1}\circ(H\otimes(\Gamma_{H_{1}}\circ(H\otimes\Gamma_{H_{1}})))\circ (((\delta_{H}\otimes\lambda_{H}^{2})\circ\delta_{H})\otimes H)$ {\footnotesize(by \eqref{mu2-exp})}
	\item[]$=\mu_{H}^{1}\circ(H\otimes(\Gamma_{H_{1}}\circ (({\rm id}_{H}\ast\lambda_{H}^{2})\otimes H)))\circ(\delta_{H}\otimes H)$ {\footnotesize(by the coassociativity of $\delta_{H}$ and \eqref{actionprod} for $(H_{1},\Gamma_{H_{1}})$)}
	\item[]$=\mu_{H}^{1}$ {\footnotesize(by \eqref{antipode} for $H_{2}$, the counit property and \eqref{actioneta} for $(H_{1},\Gamma_{H_{1}})$).}
\end{itemize}
In the case that ${\sf C}={}_{\mathbb{K}}{\sf Vect}$, the previous equalities coincide with \cite[Remark 1.9]{AGV}.

Furthermore, when ${\mathbb H}$ is cocommutative, the action $\Gamma_{H_{1}}$ is also a coalgebra morphism \cite[Lemma~2.5]{VGRRProj}, which implies that $(H_{1},\Gamma_{H_{1}})$ is a left $H_{2}$-module coalgebra.

It is also important to note that, if $\mathbb{H}$ is a cocommutative Hopf brace, $(H,\Phi_{H})$ is a right $H_{2}$-module coalgebra \cite[Lemma 2.2(1)]{AGV} with the action given by
\begin{gather*}
	\Phi_{H}\coloneqq \mu_{H}^{2}\circ ((\lambda_{H}^{2}\circ\Gamma_{H_{1}})\otimes\mu_{H}^{2})\circ(H\otimes c_{H,H}\otimes H)\circ(\delta_{H}\otimes\delta_{H}).
\end{gather*}

As was proved by Angiono, Galindo and Vendramin \cite{AGV} in the vector space setting, the category of Hopf braces is equivalent to the category of invertible 1-cocycles (see \cite[Theorem~2]{VGRRMod} for the details of this result in a general braided monoidal context), and also to the category of matched pairs over a fixed Hopf algebra $A$. In this paper, we will be interested in the construction of Hopf braces via matched pairs. Therefore, in what follows we will detail the functors involved in that categorical isomorphism in a general braided monoidal framework, whose proof is analogous to the one given by Angiono, Galindo and Vendramin \cite[Theorem~3.3]{AGV} in the category of vector spaces over a field.

\begin{definition}\label{MPdef}
	Let $H$ and $A$ be Hopf algebras in {\sf C}. A 4-tuple $(A,H,\varphi_{A},\phi_{H})$ is said to be a matched pair of Hopf algebras if the following conditions hold:
	\begin{itemize}
		\item[(i)] $(A,\varphi_{A})$ is a left $H$-module coalgebra and $(H,\phi_{H})$ is a right $A$-module coalgebra,
		\item[(ii)] $\varphi_{A}\circ (H\otimes\eta_{A})=\varepsilon_{H}\otimes\eta_{A}$, i.e., $\eta_{A}$ is a morphism of left $H$-modules, 
		\item[(iii)] $\phi_{H}\circ(\eta_{H}\otimes A)=\eta_{H}\otimes\varepsilon_{A}$, i.e., $\eta_{H}$ is a morphism of right $A$-modules,
		\item[(iv)] $\varphi_{A}\circ(H\otimes\mu_{A})=\mu_{A}\circ (A\otimes\varphi_{A})\circ (\Psi\otimes A)$,
		\item[(v)] $\phi_{H}\circ(\mu_{H}\otimes A)=\mu_{H}\circ(\phi_{H}\otimes H)\circ(H\otimes \Psi)$,
		\item[(vi)] $c_{A,H}\circ\Psi=(\phi_{H}\otimes\varphi_{A})\circ (H\otimes c_{H,A}\otimes A)\circ(\delta_{H}\otimes\delta_{A})$,
	\end{itemize}
	where \begin{equation*}\label{interw}\Psi\coloneqq (\varphi_{A}\otimes\phi_{H})\circ (H\otimes c_{H,A}\otimes A)\circ(\delta_{H}\otimes\delta_{A}).\end{equation*}
\end{definition}

Matched pairs of Hopf algebras constitute a category denoted by ${\sf MP}$, whose morphisms are given by pairs
\[(f,g)\colon (A,H,\varphi_{A},\phi_{H})\rightarrow(A',H',\varphi_{A'},\phi_{H'})\]
where $f\colon A\rightarrow A'$ and $g\colon H\rightarrow H'$ are Hopf algebra morphisms in ${\sf C}$ such that the following conditions hold:
\begin{gather}\label{morMP}
	f\circ\varphi_{A}=\varphi_{A'}\circ(g\otimes f),\qquad
	g\circ\phi_{H}=\phi_{H'}\circ(g\otimes f).
\end{gather}

Let $A$ be a cocommutative Hopf algebra in ${\sf C}$. We will consider the subcategory ${\sf MP}(A)$ of ${\sf MP}$ whose objects are those of the form $(A,A,\varphi_{A},\phi_{A})$ satisfying that
\begin{gather}\label{condinterw}
	\mu_{A}=\mu_{A}\circ\Psi,
\end{gather}
whereas the morphisms in ${\sf MP}(A)$ are given by 
\[(f,f)\colon (A,A,\varphi_{A},\phi_{A})\rightarrow(A,A,\varphi'_{A},\phi'_{A})\]
satisfying the conditions in \eqref{morMP}. To make the notation simpler, morphisms in ${\sf MP}(A)$ will be identified with a single morphism $f\equiv(f,f)$.

Hence, the theorem we are interested in asserts the following:
\begin{theorem}\label{isoHBrMP}
	Let $A$ be a cocommutative Hopf algebra in ${\sf C}$. The categories ${\sf MP}(A)$ and ${\sf cocHBr}(A)$ are isomorphic.
\end{theorem}
\begin{proof}
	By \cite[Proposition 3.1, Proposition 3.2 and Theorem 3.3]{AGV}, the required categorical isomorphism follows from the existence of the following functors:
	\begin{align*}
		{\sf F}\colon {\sf cocHBr}(A)\longrightarrow{\sf MP}(A),
	\end{align*}
	defined on objects by ${\sf F}(\mathbb{A}=(A_{1},A))=(A,A,\Gamma_{A_{1}},\Phi_{A})$, and on morphisms by the identity,~and 
	\begin{align*}
		{\sf G}\colon {\sf MP}(A)\longrightarrow{\sf cocHBr}(A),
	\end{align*}
	whose definition on objects is ${\sf G}((A,A,\varphi_{A},\phi_{A}))=\overline{\mathbb{A}}=(\overline{A},A)$, where 
	\[\overline{A}=(A,\eta_{A},\overline{\mu}_{A},\varepsilon_{A},\delta_{A},\overline{\lambda}_{A})\]
	is the Hopf algebra in ${\sf C}$ with product and antipode given by
	\begin{gather*}
		\overline{\mu}_{A}\coloneqq \mu_{A}\circ(A\otimes(\varphi_{A}\circ (\lambda_{A}\otimes A)))\circ(\delta_{A}\otimes A),\\
		\overline{\lambda}_{A}\coloneqq\varphi_{A}\circ(A\otimes\lambda_{A})\circ\delta_{A},
	\end{gather*}
	while ${\sf G}$ is the identity on morphisms. Note that, in this situation, $\Gamma_{\overline{A}}=\varphi_{A}.$
\end{proof}
\begin{remark}\label{Fgeneral}
One readily checks that the functor ${\sf F}$ is well defined without restricting to the subcategory ${\sf cocHBr}(A)$; rather, one may consider the entire subcategory of cocommutative Hopf braces without fixing any particular Hopf algebra $A$. That is, there exists a functor
\[{\sf F}\colon {\sf cocHBr}\longrightarrow {\sf MP}\]
defined on objects by ${\sf F}(\mathbb{H}=(H_{1},H_{2}))=(H_{2},H_{2},\Gamma_{H_{1}},\Phi_{H})$, and on morphisms by ${\sf F}(f)=(f,f).$ Note also that ${\sf F}(\mathbb{H}=(H_{1},H_{2}))=(H_{2},H_{2},\Gamma_{H_{1}},\Phi_{H})$ always satisfies the property \eqref{condinterw}:
\[\mu_{H}^{2}=\mu_{H}^{2}\circ \Psi,\]
where $\Psi=(\Gamma_{H_{1}}\otimes\Phi_{H})\circ(H\otimes c_{H,H}\otimes H)\circ(\delta_{H}\otimes\delta_{H})$ in this situation.

We will sometimes abuse notation by writing ${\sf F}$ for both the functor defined above and its restriction to ${\sf cocHBr}(A)$, being the latter restriction the one which provides the categorical isomorphism of Theorem \ref{isoHBrMP}.
\end{remark}
\section{Opposite brace triples and Hopf braces}

In \cite{VGRHac} the authors in collaboration with Fernández Vilaboa studied the conditions under which a triple $(H,\gamma_{H},T_{H})$, where $H$ is a Hopf algebra and $\gamma_{H}\colon H\otimes H\rightarrow H$, $T_{H}\colon H\rightarrow H$ are morphisms in ${\sf C}$, gives rise to a new Hopf algebra structure $H_{{\sf BT}}$ in such a way that the pair $\mathbb{H}_{{\sf BT}}=(H,H_{{\sf BT}})$ defines a Hopf brace in ${\sf C}$. These conditions are encoded in a new category, the category of brace triples. It is further shown that every \texttt{s}-Hopf brace (a subcategory of {\sf HBr} generalizing cocommutative Hopf braces) arises from a brace triple, and hence the two categories are isomorphic. It is worth noting that the deformed Hopf algebra $H_{{\sf BT}}$ occupies the second component of the pair of Hopf algebras that generates the Hopf brace $\mathbb{H}_{{\sf BT}}$.

Bearing in mind the above considerations and the asymmetry between the roles of $H_{1}$ and $H_{2}$ in a Hopf brace structure, one may formulate the inverse problem: given a Hopf algebra $A$ together with a pair of morphisms as above, how can one deform the structure of $A$ so as to obtain a new Hopf algebra $\widetilde{A}$ such that $(\widetilde{A},A)$ forms a Hopf brace in ${\sf C}$? In contrast with the work developed in \cite{VGRHac}, $\widetilde{A}$ here occupies the first Hopf algebra involved in the new Hopf brace. These conditions lead to a new category, the category of opposite brace triples that will be introduced hereunder.

Furthermore, we will establish categorical links between this new category and not only the category of Hopf braces, but also that of matched pairs. Such a connection is to be expected, since Theorem \ref{isoHBrMP} shows that the functor ${\sf G}$ defining the categorical isomorphism produces a Hopf algebra $\overline{A}$ functioning as $H_{1}$ within the Hopf brace $\overline{\mathbb{A}}$.

Thus, opposite brace triples are defined as follows.
\begin{definition}\label{opBTdef}
	Let $A=(A,\eta_{A},\mu_{A},\varepsilon_{A},\delta_{A},\lambda_{A})$ be a Hopf algebra in ${\sf C}$ and let $m_{A}\colon A\otimes A\rightarrow A$ and $u_{A}\colon A \rightarrow A$ be morphisms in ${\sf C}$. We will say that $(A,m_{A},u_{A})$ is an opposite brace triple in ${\sf C}$ if the following conditions hold:
	\begin{itemize}
		\item[(i)] $m_{A}$ is a coalgebra morphism,
		\item[(ii)] $m_{A}\circ(\eta_{A}\otimes A)=\mathrm{id}_{A}$,
		\item[(iii)] $m_{A}\circ(A\otimes\eta_{A})=\varepsilon_{A}\otimes\eta_{A}$,
		\item[(iv)] $m_{A}\circ(A\otimes m_{A})=m_{A}\circ(\mu_{A}^{\textnormal{op}}\otimes A)$,
		\item[(v)] $m_{A}\circ(A\otimes\widetilde{\mu}_{A})=\widetilde{\mu}_{A}\circ(m_{A}\otimes m_{A})\circ(A\otimes c_{A,A}\otimes A)\circ(\delta_{A}\otimes A\otimes A)$,
		\item[(vi)] $u_{A}$ is a coalgebra morphism,
		\item[(vii)] $u_{A}\circ u_{A}=\mathrm{id}_{A}$,
		\item[(viii)] $m_{A}\circ(A\otimes u_{A})\circ\delta_{A}=\lambda_{A}$,
	\end{itemize}
	where
	\begin{gather}\label{muopdef}
		\widetilde{\mu}_{A}\coloneqq \mu_{A}\circ(A\otimes m_{A})\circ(\delta_{A}\otimes A).
	\end{gather}
	
	These objects give rise to a category whose morphisms are defined as follows: if $(A,m_{A},u_{A})$ and $(B,m_{B},u_{B})$ are opposite brace triples, then $$f\colon (A,m_{A},u_{A})\rightarrow (B,m_{B},u_{B})$$ is a morphism of opposite brace triples if $f\colon A\rightarrow B$ is a morphism of Hopf algebras in ${\sf C}$ and the condition 
	\begin{gather}\label{opBTmorcond}
		f\circ m_{A}=m_{B}\circ(f\otimes f)
	\end{gather}
	holds. This category will be denoted by ${\sf opBT}.$
	
	In case that the underlying Hopf algebra $A$ is cocommutative, then $(A,m_{A},u_{A})$ is said to be a cocommutative opposite brace triple. Cocommutative opposite brace triples form a full subcategory of ${\sf opBT}$ which we will denote by ${\sf coc}\textnormal{-}{\sf opBT}$.
\end{definition}
\begin{remark}\label{opjust}
	In the case that $(A,m_{A},u_{A})$ is a cocommutative opposite brace triple, conditions (ii) and (iv) in Definition \ref{opBTdef} mean that $(A,m_{A})$ is a left $A^{{\rm op}}$-module in ${\sf C}$. This fact provides the justification for the name assigned to the previous category, employing the adjective ``opposite''. 
	
	In a more general setting, without assuming cocommutativity, if $(A,m_{A},u_{A})$ is an opposite brace triple such that $\lambda_{A}$ is an isomorphism in ${\sf C}$, then $(A,m_{A})$ is a left $A^{{\rm op}}$-module in $\overline{{\sf C}}$.
\end{remark}

Our starting point is the construction of a new Hopf algebra structure $\widetilde{A}$ derived from an opposite brace triple $(A,m_{A},u_{A})$.
\begin{theorem}\label{conswidetildeA}
	If $(A,m_{A},u_{A})$ is a cocommutative opposite brace triple, then 
	\[\widetilde{A}=(A,\eta_{A},\widetilde{\mu}_{A},\varepsilon_{A},\delta_{A},u_{A})\]
	is a Hopf algebra in ${\sf C}$.
\end{theorem}
\begin{proof}
	Let us start the proof by showing that the triple $(A,\eta_{A},\widetilde{\mu}_{A})$ is an algebra in ${\sf C}$. The unit property follows by
	\begin{itemize}
		\itemindent=-27pt
		\item[] $\hspace{0.38cm} \widetilde{\mu}_{A}\circ(\eta_{A}\otimes A)$
		\item[] $=\mu_{A}\circ(\eta_{A}\otimes(m_{A}\circ(\eta_{A}\otimes A)))$ {\footnotesize (by \eqref{muopdef} and the condition of coalgebra morphism for $\eta_{A}$)}
		\item[] $=m_{A}\circ(\eta_{A}\otimes A)$ {\footnotesize (by the unit property)}
		\item[] $=\mathrm{id}_{A}$ {\footnotesize (by (ii) of Definition \ref{opBTdef}),} 
	\end{itemize}
	and 
	\begin{itemize}
		\itemindent=-27pt
		\item[] $\hspace{0.38cm} \widetilde{\mu}_{A}\circ(A\otimes\eta_{A})$
		\item[] $=\mu_{A}\circ(A\otimes(\eta_{A}\circ\varepsilon_{A}))\circ\delta_{A}$ {\footnotesize (by \eqref{muopdef} and (iii) of Definition \ref{opBTdef})}
		\item[] $=\mathrm{id}_{A}$ {\footnotesize (by the (co)unit property).}
	\end{itemize}
	
	In order to compute the associativity of $\widetilde{\mu}_{A}$ we will see first that $\widetilde{\mu}_{A}$ is a coalgebra morphism. The fact that $\varepsilon_{A}\circ\widetilde{\mu}_{A}=\varepsilon_{A}\otimes\varepsilon_{A}$ is straightforward by the fact that $\mu_{A}$ and $m_{A}$ preserve the counit of $A$. Moreover,
	\begin{itemize}
		\itemindent=-27pt
		\item[] $\hspace{0.38cm} \delta_{A}\circ\widetilde{\mu}_{A}$
		\item[] $=(\mu_{A}\otimes\mu_{A})\circ(A\otimes c_{A,A}\otimes A)\circ(\delta_{A}\otimes(\delta_{A}\circ m_{A}))\circ(\delta_{A}\otimes A)$ {\footnotesize (by the condition of coalgebra morphism}
		\item[] $\hspace{0.38cm}${\footnotesize for $\mu_{A}$)}
		\item[] $=(\mu_{A}\otimes\mu_{A})\circ(A\otimes c_{A,A}\otimes A)\circ (\delta_{A}\otimes((m_{A}\otimes m_{A})\circ(A\otimes c_{A,A}\otimes A)\circ(\delta_{A}\otimes\delta_{A})))\circ(\delta_{A}\otimes A)$ 
		\item[] $\hspace{0.38cm}${\footnotesize (by (i) of Definition \ref{opBTdef})}
		\item[] $=((\mu_{A}\circ(A\otimes m_{A}))\otimes(\mu_{A}\circ(A\otimes m_{A})))\circ(A\otimes((A\otimes c_{A,A})\circ((c_{A,A}\circ\delta_{A})\otimes A))\otimes A\otimes A)$
		\item[] $\hspace{0.38cm}\circ(\delta_{A}\otimes c_{A,A}\otimes A)\circ(\delta_{A}\otimes\delta_{A})$ {\footnotesize (by naturality of $c$ and the coassociativity of $\delta_{A}$)}
		\item[] $=((\mu_{A}\circ(A\otimes m_{A}))\otimes(\mu_{A}\circ(A\otimes m_{A})))\circ(A\otimes A\otimes((c_{A,A}\otimes A)\circ(A\otimes c_{A,A})\circ(\delta_{A}\otimes A))\otimes A)$
		\item[] $\hspace{0.38cm}\circ(((\delta_{A}\otimes A)\circ\delta_{A})\otimes\delta_{A})$ {\footnotesize (by the cocommutativity of $A$ and the coassociativity of $\delta_{A}$)}
		\item[] $=(\widetilde{\mu}_{A}\otimes\widetilde{\mu}_{A})\circ(A\otimes c_{A,A}\otimes A)\circ(\delta_{A}\otimes\delta_{A})$ {\footnotesize (by the naturality of $c$ and \eqref{muopdef}).}
	\end{itemize}
	
	As a consequence, the associativity of the product $\widetilde{\mu}_{A}$ follows from
	\begin{itemize}
		\itemindent=-27pt
		\item[] $\hspace{0.38cm} \widetilde{\mu}_{A}\circ(A\otimes\widetilde{\mu}_{A})$
		\item[] $=\mu_{A}\circ(A\otimes(\widetilde{\mu}_{A}\circ(m_{A}\otimes m_{A})\circ(A\otimes c_{A,A}\otimes A)\circ(\delta_{A}\otimes A\otimes A)))\circ(\delta_{A}\otimes A\otimes A)$ {\footnotesize (by \eqref{muopdef} and}
		\item[] $\hspace{0.38cm}${\footnotesize (v) of Definition \ref{opBTdef})}
		\item[] $=\mu_{A}\circ(A\otimes (\mu_{A}\circ (A\otimes m_{A})\circ (((m_{A}\otimes m_{A})\circ(A\otimes c_{A,A}\otimes A)\circ(\delta_{A}\otimes \delta_{A}))\otimes m_{A})$
		\item[] $\hspace{0.38cm}\circ (A\otimes c_{A,A}\otimes A)\circ(\delta_{A}\otimes A\otimes A)))\circ (\delta_{A}\otimes A\otimes A)$ {\footnotesize (by \eqref{muopdef} and (i) of Definition \ref{opBTdef})}
		\item[] $=\mu_{A}\circ (\widetilde{\mu}_{A}\otimes(m_{A}\circ((\mu_{A}\circ c_{A,A}\circ (m_{A}\otimes A)\circ(A\otimes c_{A,A})\circ (\delta_{A}\otimes A))\otimes A)))$
		\item[] $\hspace{0.38cm}\circ (((A\otimes c_{A,A}\otimes A)\circ(\delta_{A}\otimes\delta_{A}))\otimes A)$ {\footnotesize (by the naturality of $c$, the coassociativity of $\delta_{A}$, the associativity}
		\item[] $\hspace{0.38cm}${\footnotesize of $\mu_{A}$, \eqref{muopdef} and (iv) of Definition \ref{opBTdef})}
		\item[] $=\mu_{A}\circ (\widetilde{\mu}_{A}\otimes(m_{A}\circ((\mu_{A}\circ (A\otimes m_{A})\circ(c_{A,A}\otimes A)\circ(A\otimes (c_{A,A}\circ c_{A,A})) \circ (\delta_{A}\otimes A))\otimes A)))$
		\item[] $\hspace{0.38cm}\circ (((A\otimes c_{A,A}\otimes A)\circ(\delta_{A}\otimes\delta_{A}))\otimes A)$ {\footnotesize (by the naturality of $c$)}
		\item[] $=\mu_{A}\circ (\widetilde{\mu}_{A}\otimes(m_{A}\circ((\mu_{A}\circ (A\otimes m_{A})\circ((c_{A,A}\circ\delta_{A})\otimes A))\otimes A)))\circ (((A\otimes c_{A,A}\otimes A)\circ(\delta_{A}\otimes\delta_{A}))\otimes A)$
		\item[]$\hspace{0.38cm}${\footnotesize(by \eqref{ccb})}
		\item[] $=\mu_{A}\circ(A\otimes m_{A})\circ(((\widetilde{\mu}_{A}\otimes\widetilde{\mu}_{A})\circ(A\otimes c_{A,A}\otimes A)\circ(\delta_{A}\otimes\delta_{A}))\otimes A)$ {\footnotesize (by the cocommutativity of $A$}
		\item[] $\hspace{0.38cm}${\footnotesize and \eqref{muopdef})}
		\item[] $=\widetilde{\mu}_{A}\circ(\widetilde{\mu}_{A}\otimes A)$ {\footnotesize (by the condition of coalgebra morphism for $\widetilde{\mu}_{A}$ and \eqref{muopdef}).}
	\end{itemize}
	
	Then, we conclude that $(A,\eta_{A},\widetilde{\mu}_{A})$ is an algebra in ${\sf C}$, and also, due to the fact that $\widetilde{\mu}_{A}$ is a coalgebra morphism, $(A,\eta_{A},\widetilde{\mu}_{A},\varepsilon_{A},\delta_{A})$ is a bialgebra in ${\sf C}$. 
	
	Therefore, to show that $\widetilde{A}$ is a Hopf algebra in ${\sf C}$ it only remains to prove that $u_{A}$ satisfies \eqref{antipode}. We will denote by $\widetilde{\ast}$ the convolution product in $\operatorname{Hom}_{{\sf C}}(A,\widetilde{A})$. On the one hand,
	\begin{itemize}
		\itemindent=-27pt
		\item[] $\hspace{0.38cm} \mathrm{id}_{A}\,\widetilde{\ast}\, u_{A}$
		\item[] $=\mu_{A}\circ(A\otimes (m_{A}\circ(A\otimes u_{A})\circ\delta_{A}))\circ\delta_{A}$ {\footnotesize (by \eqref{muopdef} and coassociativity of $\delta_{A}$)}
		\item[] $=\mathrm{id}_{A}\ast\lambda_{A}$ {\footnotesize (by (viii) of Definition \ref{opBTdef})}
		\item[] $=\varepsilon_{A}\otimes\eta_{A}$ {\footnotesize (by \eqref{antipode} for the Hopf algebra $A$)}
	\end{itemize}
	and, on the other hand,
	\begin{itemize}
		\itemindent=-27pt
		\item[] $\hspace{0.38cm}u_{A}\,\widetilde{\ast}\, \mathrm{id}_{A}$
		\item[] $=\widetilde{\mu}_{A}\circ(u_{A}\otimes(u_{A}\circ u_{A}))\circ\delta_{A}$ {\footnotesize (by (vii) of Definition \ref{opBTdef})}
		\item[] $=(\mathrm{id}_{A}\,\widetilde{\ast}\, u_{A})\circ u_{A}$ {\footnotesize (by (vi) of Definition \ref{opBTdef})}
		\item[] $=\eta_{A}\circ\varepsilon_{A}\circ u_{A}$ {\footnotesize (by the equality above: $\mathrm{id}_{A}\,\widetilde{\ast}\, u_{A}=\varepsilon_{A}\otimes\eta_{A}=\eta_{A}\circ\varepsilon_{A}$)}
		\item[] $=\varepsilon_{A}\otimes\eta_{A}$ {\footnotesize (by (vi) of Definition \ref{opBTdef}),}
	\end{itemize}
	which implies that $\widetilde{A}$ is a Hopf algebra in ${\sf C}$.
\end{proof}
\begin{corollary}
	If $(A,m_{A},u_{A})$ is a cocommutative opposite brace triple, then $(\widetilde{A},m_{A})$ is a left $A^{{\rm op}}$-module algebra-coalgebra in ${\sf C}$.
\end{corollary}
\begin{proof}
	By Remark \ref{opjust}, this proof is a direct consequence of the previous theorem and axioms (i), (iii) and (v) in Definition \ref{opBTdef}.
\end{proof}

Hence, given a cocommutative opposite brace triple $(A,m_{A},u_{A})$, Theorem \ref{conswidetildeA} constructs a new Hopf algebra in ${\sf C}$, $\widetilde{A}$, whose underlying coalgebra is the same as that of $A$. Therefore, one may naturally ask whether the pair $(\widetilde{A},A)$ is a Hopf brace in ${\sf C}$. The following results give a positive answer to the previous question.

\begin{lemma}
	If $(A,m_{A},u_{A})$ is a cocommutative opposite brace triple, then the equality
	\begin{gather}\label{muAitomuop}
		\mu_{A}=\widetilde{\mu}_{A}\circ(A\otimes(m_{A}\circ (\lambda_{A}\otimes A)))\circ (\delta_{A}\otimes A)
	\end{gather}
	holds.
\end{lemma}
\begin{proof}
	The equality above results as follows:
	\begin{itemize}
		\itemindent=-27pt
		\item[] $\hspace{0.38cm}\widetilde{\mu}_{A}\circ(A\otimes(m_{A}\circ (\lambda_{A}\otimes A)))\circ (\delta_{A}\otimes A)$
		\item[] $=\mu_{A}\circ(A\otimes(m_{A}\circ(A\otimes(m_{A}\circ(\lambda_{A}\otimes A)))\circ (\delta_{A}\otimes A)))\circ(\delta_{A}\otimes A)$ {\footnotesize (by \eqref{muopdef} and the coassociativity}
		\item[] $\hspace{0.38cm}${\footnotesize of $\delta_{A}$)}
		\item[] $=\mu_{A}\circ(A\otimes(m_{A}\circ((\mu_{A}\circ c_{A,A}\circ (A\otimes\lambda_{A})\circ\delta_{A})\otimes A)))\circ(\delta_{A}\otimes A)$ {\footnotesize (by (iv) of Definition~\ref{opBTdef})}
		\item[] $=\mu_{A}\circ(A\otimes(m_{A}\circ((\lambda_{A}\ast\mathrm{id}_{A})\otimes A)))\circ(\delta_{A}\otimes A)$ {\footnotesize (by the naturality of $c$ and the cocommutativity}
		\item[] $\hspace{0.38cm}${\footnotesize of $\delta_{A}$)}
		\item[] $=\mu_{A}\circ(A\otimes(m_{A}\circ((\eta_{A}\circ\varepsilon_{A})\otimes A)))\circ(\delta_{A}\otimes A)$ {\footnotesize (by \eqref{antipode} for the Hopf algebra $A$)}
		\item[] $=\mu_{A}$ {\footnotesize (by the counit property and (ii) of Definition \ref{opBTdef}).}\qedhere
	\end{itemize}
\end{proof}
\begin{theorem}\label{HBrwidetildeA}
	If $(A,m_{A},u_{A})$ is a cocommutative opposite brace triple, then
	\[\widetilde{\mathbb{A}}=(\widetilde{A},A)\]
	is a cocommutative Hopf brace in ${\sf C}$.
\end{theorem}
\begin{proof}
	The proof is complete once we verify that \eqref{compatHbrace} holds for the pair $(\widetilde{A},A)$. Note first that
	\begin{gather}\label{gammaBTop}
		\Gamma_{\widetilde{A}}=m_{A}\circ(\lambda_{A}\otimes A).
	\end{gather}
	
	Indeed,
	\begin{itemize}
		\itemindent=-27pt
		\item[] $\hspace{0.38cm}\Gamma_{\widetilde{A}}$
		\item[] $=\widetilde{\mu}_{A}\circ(u_{A}\otimes\mu_{A})\circ(\delta_{A}\otimes A)$ {\footnotesize (by \eqref{def_GammaH1})}
		\item[] $=\widetilde{\mu}_{A}\circ(u_{A}\otimes(\widetilde{\mu}_{A}\circ(A\otimes (m_{A}\circ(\lambda_{A}\otimes A)))\circ(\delta_{A}\otimes A)))\circ(\delta_{A}\otimes A)$ {\footnotesize (by \eqref{muAitomuop})}
		\item[] $=\widetilde{\mu}_{A}\circ((u_{A}\,\widetilde{\ast}\, \mathrm{id}_{A})\otimes(m_{A}\circ(\lambda_{A}\otimes A)))\circ(\delta_{A}\otimes A)$ {\footnotesize (by the associativity of $\widetilde{\mu}_{A}$ and the coassociativity}
		\item[] $\hspace{0.38cm}${\footnotesize of $\delta_{A}$)}
		\item[] $=m_{A}\circ(\lambda_{A}\otimes A)$ {\footnotesize (by \eqref{antipode} for $\widetilde{A}$ and the (co)unit property).}
	\end{itemize}
	
	Therefore, 
	\begin{itemize}
		\itemindent=-27pt
		\item[] $\hspace{0.38cm}\widetilde{\mu}_{A}\circ(\mu_{A}\otimes \Gamma_{\widetilde{A}})\circ(A\otimes c_{A,A}\otimes A)\circ(\delta_{A}\otimes A\otimes A)$
		\item[] $=\mu_{A}\circ(A\otimes m_{A})\circ((\delta_{A}\circ\mu_{A})\otimes(m_{A}\circ(\lambda_{A}\otimes A)))\circ(A\otimes c_{A,A}\otimes A)\circ(\delta_{A}\otimes A\otimes A)$ 
		\item[] $\hspace{0.38cm}${\footnotesize (by \eqref{muopdef} and \eqref{gammaBTop})}
		\item[] $=\mu_{A}\circ(A\otimes m_{A})\circ(((\mu_{A}\otimes\mu_{A})\circ(A\otimes c_{A,A}\otimes A)\circ(\delta_{A}\otimes\delta_{A}))\otimes(m_{A}\circ(\lambda_{A}\otimes A)))\circ(A\otimes c_{A,A}\otimes A)$
		\item[] $\hspace{0.29cm}\circ(\delta_{A}\otimes A\otimes A)$ {\footnotesize (by the condition of coalgebra morphism for $\mu_{A}$)}
		\item[] $=\mu_{A}\circ(\mu_{A}\otimes (m_{A}\circ ((\mu_{A}\circ c_{A,A}\circ (\mu_{A}\otimes\lambda_{A})\circ(A\otimes c_{A,A})\circ(\delta_{A}\otimes A))\otimes A)))$
		\item[] $\hspace{0.38cm}\circ(((A\otimes c_{A,A}\otimes A)\circ(\delta_{A}\otimes\delta_{A}))\otimes A)$ {\footnotesize (by the naturality of $c$, the coassociativity of $\delta_{A}$ and (iv) of}
		\item[] $\hspace{0.38cm}${\footnotesize Definition \ref{opBTdef})}
		\item[] $=\mu_{A}\circ(\mu_{A}\otimes (m_{A}\circ ((\mu_{A}\circ ((\mu_{A}\circ(\lambda_{A}\otimes A)\circ c_{A,A})\otimes A)\circ(A\otimes (c_{A,A}\circ c_{A,A}))\circ(\delta_{A}\otimes A))\otimes A)))$
		\item[] $\hspace{0.38cm}\circ(((A\otimes c_{A,A}\otimes A)\circ(\delta_{A}\otimes\delta_{A}))\otimes A)$ {\footnotesize (by the naturality of $c$ and the associativity of $\mu_{A}$)}
		\item[] $=\mu_{A}\circ(\mu_{A}\otimes (m_{A}\circ ((\mu_{A}\circ ((\lambda_{A}\ast\mathrm{id}_{A})\otimes A))\otimes A)))\circ(((A\otimes c_{A,A}\otimes A)\circ(\delta_{A}\otimes\delta_{A}))\otimes A)$ 
		\item[] $\hspace{0.38cm}${\footnotesize (by \eqref{ccb} and cocommutativity of $\delta_{A}$)}
		\item[] $=\mu_{A}\circ(\mu_{A}\otimes m_{A})\circ(A\otimes\delta_{A}\otimes A)$ {\footnotesize (by \scalebox{0.85}{\eqref{antipode}} for $A$, the naturality of $c$ and the (co)unit property)}
		\item[] $=\mu_{A}\circ(A\otimes\widetilde{\mu}_{A})$ {\footnotesize (by the associativity of $\mu_{A}$ and \eqref{muopdef}),}
	\end{itemize}
	which implies that $\widetilde{\mathbb{A}}$ is a Hopf brace in ${\sf C}$.
\end{proof}

The previous theorem can be interpreted in a functorial way as follows: There exists a functor 
\[{\sf P}\colon {\sf coc}\textnormal{-}{\sf opBT}\longrightarrow{\sf cocHBr}\]
defined on objects by ${\sf P}((A,m_{A},u_{A}))=\widetilde{\mathbb{A}}=(\widetilde{A},A)$, and on morphisms by the identity because, if $f\colon (A,m_{A},u_{A})\rightarrow(B,m_{B},u_{B})\in{\sf coc}\textnormal{-}{\sf opBT}$, $f\colon \widetilde{A}\rightarrow\widetilde{B}$ is also a Hopf algebra morphism in ${\sf C}$. The previous claim follows once it is shown that $f$ preserves the deformed product:
\begin{gather}\label{fwidetildeprod}
	f\circ \widetilde{\mu}_{A}=\widetilde{\mu}_{B}\circ(f\otimes f).
\end{gather}
Indeed,
\begin{itemize}
	\itemindent=-27pt
	\item[] $\hspace{0.38cm}f\circ \widetilde{\mu}_{A}$
	\item[] $=\mu_{B}\circ(f\otimes(f\circ m_{A}))\circ(\delta_{A}\otimes A)$ {\footnotesize (by \eqref{muopdef} and the condition of algebra morphism for $f\colon A\rightarrow B$)}
	\item[] $=\mu_{B}\circ(B\otimes m_{B})\circ(((f\otimes f)\circ\delta_{A})\otimes f)$ {\footnotesize (by \eqref{opBTmorcond})}
	\item[] $=\widetilde{\mu}_{B}\circ(f\otimes f)$ {\footnotesize (by the condition of coalgebra morphism for $f$ and \eqref{muopdef}).}
\end{itemize} 

Furthermore, the following corollary arises as an immediate consequence of the preceding result.
\begin{corollary}
	If $(A,m_{A},u_{A})$ is a cocommutative opposite brace triple, then $$(\widetilde{A},m_{A}\circ(\lambda_{A}\otimes A))$$ is a left $A$-module algebra-coalgebra in ${\sf C}$.
\end{corollary}
\begin{proof}
	It is enough to apply the fact that, given any cocommutative Hopf brace $\mathbb{H}=(H_{1},H_{2})$ in ${\sf C}$, $(H_{1},\Gamma_{H_{1}})$ is a left $H_{2}$-module algebra-coalgebra to the particular Hopf brace $\widetilde{\mathbb{A}}$ constructed in Theorem \ref{HBrwidetildeA} because, in this case, $\Gamma_{\widetilde{A}}=m_{A}\circ(\lambda_{A}\otimes A)$ as was shown in \eqref{gammaBTop}.
\end{proof}
\begin{corollary}
	If $f\colon (A,m_{A},u_{A})\rightarrow(B,m_{B},u_{B})$ is a morphism of cocommutative opposite brace triples in ${\sf C}$, then 
	\begin{gather}\label{fcommutesu}
		u_{B}\circ f=f\circ u_{A}.
	\end{gather}
\end{corollary}
\begin{proof}
	In view of the fact that, as established in \eqref{fwidetildeprod}, $f\colon\widetilde{A}\rightarrow\widetilde{B}$ is a Hopf algebra morphism in ${\sf C}$, the equality \eqref{fcommutesu} follows immediately from \eqref{morant}.
\end{proof}

In order to obtain a categorical equivalence between ${\sf HBr}$ and ${\sf opBT}$ in the cocommutative setting, we need to show the existence of an inverse for the functor ${\sf P}$, which will be a consequence of the following theorem.
\begin{theorem}
	If $\mathbb{H}=(H_{1},H_{2})$ is a cocommutative Hopf brace in ${\sf C}$, then 
	\[(H_{2},m_{H_{2}}=\Gamma_{H_{1}}\circ(\lambda_{H}^{2}\otimes H),u_{H_{2}}=\lambda_{H}^{1})\]
	is a cocommutative opposite brace triple in ${\sf C}$.
\end{theorem}
\begin{proof}
	When $\mathbb{H}$ is cocommutative, the action $\Gamma_{H_{1}}$ is a coalgebra morphism. In addition, the cocommutativity of $\mathbb{H}$ implies that $H_{2}$ is a cocommutative Hopf algebra, and then, the antipode $\lambda_{H}^{2}$ is also a coalgebra morphism. Hence, $m_{H_{2}}$ is a coalgebra morphism because it is a composition of coalgebra morphisms, so condition (i) of Definition \ref{opBTdef} holds.
	
	Condition (ii) follows by \eqref{u-antip1} for $\lambda_{H}^{2}$ and \eqref{actioneta} for the left $H_{2}$-module $(H_{1},\Gamma_{H_{1}})$.
	
	Regarding condition (iii) of Definition \ref{opBTdef}, it is also straightforward by using \eqref{Hmodalg1} for the left $H_{2}$-module algebra $(H_{1},\Gamma_{H_{1}})$ and \eqref{u-antip2} for $\lambda_{H}^{2}.$
	
	Condition (iv) for $(H_{2},m_{H_{2}}=\Gamma_{H_{1}}\circ(\lambda_{H}^{2}\otimes H),u_{H_{2}}=\lambda_{H}^{1})$ results as follows:
	\begin{itemize}
		\itemindent=-27pt
		\item[] $\hspace{0.38cm}m_{H_{2}}\circ(H\otimes m_{H_{2}})$
		\item[]$=\Gamma_{H_{1}}\circ((\mu_{H}^{2}\circ (\lambda_{H}^{2}\otimes\lambda_{H}^{2}))\otimes H)$ {\footnotesize (by \eqref{actionprod} for $(H_{1},\Gamma_{H_{1}})$)}
		\item[]$=\Gamma_{H_{1}}\circ((\mu_{H}^{2}\circ (\lambda_{H}^{2}\otimes\lambda_{H}^{2})\circ c_{H,H}\circ c_{H,H})\otimes H)$ {\footnotesize (by cocommutativity of $\mathbb{H}$ and \eqref{ccb})}
		\item[]$=\Gamma_{H_{1}}\circ((\lambda_{H}^{2}\circ\mu_{H}^{2}\circ c_{H,H})\otimes H)$ {\footnotesize (by \eqref{a-antip1} for $\lambda_{H}^{2}$)}
		\item[]$=m_{H_{2}}\circ((\mu_{H}^{2})^{{\rm op}}\otimes H)$. 
	\end{itemize}
	
	In order to compute condition (v), firstly note that 
	\begin{gather}\label{muwidetildeH2}
		\widetilde{\mu}_{H_{2}}=\mu_{H}^{1}
	\end{gather}
	by equality \eqref{mu1-exp}. Therefore,
	\begin{itemize}
		\itemindent=-27pt
		\item[] $\hspace{0.38cm}m_{H_{2}}\circ(H\otimes\widetilde{\mu}_{H_{2}})$
		\item[]$=\Gamma_{H_{1}}\circ(\lambda_{H}^{2}\otimes\mu_{H}^{1})$ {\footnotesize (by definition of $m_{H_{2}}$ and \eqref{muwidetildeH2})}
		\item[]$=\mu_{H}^{1}\circ(\Gamma_{H_{1}}\otimes\Gamma_{H_{1}})\circ(H\otimes c_{H,H}\otimes H)\circ((\delta_{H}\circ\lambda_{H}^{2})\otimes H\otimes H)$ {\footnotesize(by \eqref{Hmodalg2} for $(H_{1},\Gamma_{H_{1}})$)}
		\item[]$=\mu_{H}^{1}\circ((\Gamma_{H_{1}}\circ(\lambda_{H}^{2}\otimes H))\otimes(\Gamma_{H_{1}}\circ(\lambda_{H}^{2}\otimes H)))\circ(H\otimes c_{H,H}\otimes H)\circ(\delta_{H}\otimes H\otimes H)$ {\footnotesize (by the}
		\item[]$\hspace{0.38cm}${\footnotesize  condition of coalgebra morphism for $\lambda_{H}^{2}$ under cocommutativity and naturality of $c$)}
		\item[] $=\widetilde{\mu}_{H_{2}}\circ(m_{H_{2}}\otimes m_{H_{2}})\circ (H\otimes c_{H,H}\otimes H)\circ(\delta_{H}\otimes H\otimes H)$ {\footnotesize (by definition of $m_{H_{2}}$ and \eqref{muwidetildeH2}).}
	\end{itemize}
	
	By the same argument used for $H_{2}$, since $\mathbb{H}$ is a cocommutative Hopf brace, it follows that $H_{1}$ is likewise cocommutative. Consequently, its antipode $\lambda_{H}^{1}$ is an involutive coalgebra morphism. Hence, axioms (vi) and (vii) of Definition \ref{opBTdef} are satisfied.
	
	To conclude, the remaining condition (viii) of Definition \ref{opBTdef} for the triple $(H_{2},m_{H_{2}}=\Gamma_{H_{1}}\circ(\lambda_{H}^{2}\otimes H),u_{H_{2}}=\lambda_{H}^{1})$ is obtained as follows:
	\begin{itemize}
		\itemindent=-27pt
		\item[] $\hspace{0.38cm}\Gamma_{H_{1}}\circ(\lambda_{H}^{2}\otimes\lambda_{H}^{1})\circ\delta_{H}$
		\item[] $=\mu_{H}^{1}\circ((\lambda_{H}^{1}\circ\mu_{H}^{2})\otimes H)\circ(H\otimes c_{H,H})\circ((\delta_{H}\circ\lambda_{H}^{2})\otimes H)\circ\delta_{H}$ {\footnotesize (by \eqref{agv1})}
		\item[] $=\mu_{H}^{1}\circ((\lambda_{H}^{1}\circ\mu_{H}^{2})\otimes H)\circ(H\otimes c_{H,H})\circ((c_{H,H}\circ(\lambda_{H}^{2}\otimes\lambda_{H}^{2})\circ\delta_{H})\otimes H)\circ\delta_{H}$ {\footnotesize (by \eqref{a-antip2} for $\lambda_{H}^{2}$)}
		\item[] $=\mu_{H}^{1}\circ(\lambda_{H}^{1}\otimes H)\circ c_{H,H}\circ(\lambda_{H}^{2}\otimes (\lambda_{H}^{2}\ast\mathrm{id}_{H}))\circ\delta_{H}$ {\footnotesize (by the coassociativity of $\delta_{H}$ and the naturality}
		\item[] $\hspace{0.38cm}${\footnotesize of $c$)}
		\item[] $=\lambda_{H}^{2}$ {\footnotesize (by \eqref{antipode} for $H_{2}$, the counit property, the naturality of $c$, \eqref{u-antip1} for $\lambda_{H}^{1}$ and the unit property).}\qedhere
	\end{itemize}
\end{proof}

Consequently, there exists a functor 
\[{\sf Q}\colon {\sf cocHBr}\longrightarrow{\sf coc}\textnormal{-}{\sf opBT}\]
acting on objects by ${\sf Q}(\mathbb{H})=(H_{2},m_{H_{2}}=\Gamma_{H_{1}}\circ(\lambda_{H}^{2}\otimes H),u_{H_{2}}=\lambda_{H}^{1})$, and on morphisms by the identity. In fact, if $f\colon \mathbb{H}=(H_{1},H_{2})\rightarrow\mathbb{B}=(B_{1},B_{2})$ is a morphism of cocommutative Hopf braces, then $f\colon (H_{2},m_{H_{2}},u_{H_{2}})\rightarrow(B_{2},m_{B_{2}},u_{B_{2}})$ is a morphism in ${\sf coc}\textnormal{-}{\sf opBT}$. It suffices to show that 
\begin{gather}\label{fmmor}f\circ m_{H_{2}}=m_{B_{2}}\circ(f\otimes f).\end{gather}
Indeed, by a direct computation, one can easily prove that, in the previous situation, the following equality
\begin{gather}\label{fGammacompat}
	f\circ\Gamma_{H_{1}}=\Gamma_{B_{1}}\circ(f\otimes f)
\end{gather}
holds. Therefore, \eqref{fmmor} is obtained as follows:
\begin{itemize}
	\itemindent=-27pt
	\item[] $\hspace{0.38cm}f\circ m_{H_{2}}$
	\item[]$=f\circ\Gamma_{H_{1}}\circ(\lambda_{H}^{2}\otimes H)$
	\item[]$=\Gamma_{B_{1}}\circ((f\circ \lambda_{H}^{2})\otimes f)$ {\footnotesize (by \eqref{fGammacompat})}
	\item[]$=\Gamma_{B_{1}}\circ((\lambda_{B}^{2}\circ f)\otimes f)$ {\footnotesize (by \eqref{morant} for the Hopf algebra morphism $f\colon H_{2}\rightarrow B_{2}$)}
	\item[] $=m_{B_{2}}\circ(f\otimes f).$
\end{itemize}
\begin{example}
	The composition of the functor ${\sf G}$ given in Theorem \ref{isoHBrMP} and the functor ${\sf Q}$ above provides a systematic mechanism for the construction of examples of opposite brace triples based on matched pairs of Hopf algebras: if $(A,A,\varphi_{A},\phi_{A})$ is an object in ${\sf MP}(A)$, where $A$ is a cocommutative Hopf algebra in ${\sf C}$, then
	\[(A, \varphi_{A}\circ(\lambda_{A}\otimes A), \varphi_{A}\circ(A\otimes\lambda_{A})\circ\delta_{A})\]
	is an opposite brace triple in ${\sf C}.$
	
	Indeed, every matched pair in the essential image of ${\sf F}$ (cf. Remark \ref{Fgeneral}) falls into the above situation. 
\end{example}

We now reach the main theorem of this paper, in which we prove that the functors ${\sf P}$ and ${\sf Q}$ induce the desired categorical isomorphism.
\begin{theorem}\label{isomain}
	The categories ${\sf cocHBr}$ and ${\sf coc}\textnormal{-}{\sf opBT}$ are isomorphic.
\end{theorem}
\begin{proof}
	If $\mathbb{H}=(H_{1},H_{2})$ is a cocommutative Hopf brace in ${\sf C}$, then the following is obtained
	\[({\sf P}\circ{\sf Q})(\mathbb{H})={\sf P}((H_{2},\Gamma_{H_{1}}\circ(\lambda_{H}^{2}\otimes H),\lambda_{H}^{1}))=(\widetilde{H_{2}},H_{2}),\]
	where $\widetilde{H_{2}}$ is a Hopf algebra whose product is given by
	\[\widetilde{\mu}_{H_{2}}=\mu_{H}^{2}\circ(H\otimes (\Gamma_{H_{1}}\circ(\lambda_{H}^{2}\otimes H)))\circ(\delta_{H}\otimes H)=\mu_{H}^{1},\]
	the last equality being a consequence of \eqref{mu1-exp}. Then, we obtain that $\widetilde{H_{2}}=H_{1}$ by the uniqueness of the antipode for a Hopf algebra structure, and this implies that $({\sf P}\circ{\sf Q})(\mathbb{H})=\mathbb{H}$, i.e., ${\sf P}\circ{\sf Q}={\sf id}_{{\sf cocHBr}}$.
	
	On the other hand, if $(A,m_{A},u_{A})$ is a cocommutative opposite brace triple, then we deduce that
	\[({\sf Q}\circ{\sf P})((A,m_{A},u_{A}))={\sf Q}(\widetilde{\mathbb{A}})=(A,\Gamma_{\widetilde{A}}\circ(\lambda_{A}\otimes A),u_{A})=(A,m_{A},u_{A}),\]
	where the last equality is a consequence of \eqref{gammaBTop} and \eqref{lambdasquareid}. Then, ${\sf Q}\circ{\sf P}={\sf id}_{{\sf coc}\textnormal{-}{\sf opBT}},$ which concludes the proof.
\end{proof}
Given a cocommutative Hopf algebra $A$ in ${\sf C}$, let us denote by ${\sf coc}\textnormal{-}{\sf opBT}(A)$ the full subcategory of cocommutative opposite brace triples over the fixed Hopf algebra $A$. The following corollary is directly obtained.
\begin{corollary}\label{coroMP}
	Let $A$ be a cocommutative Hopf algebra in ${\sf C}$. The categories ${\sf cocHBr}(A)$, ${\sf MP}(A)$ and ${\sf coc}\textnormal{-}{\sf opBT}(A)$ are isomorphic.
\end{corollary}
\begin{proof}
	It is straightforward by Theorems \ref{isoHBrMP} and \ref{isomain}.
\end{proof}

\section*{Funding}

The  authors are supported by  Ministerio de Ciencia e Innovaci\'on. Agencia Estatal de Investigaci\'on (Spain) grant no. PID2024-15502NB-I00 (European FEDER support included, UE).

Moreover, B. Ramos Pérez is funded by Xunta de Galicia through the Competitive Reference Groups (GRC) grant no. ED431C 2023/31,  and the fellowship grant no. ED481A-2023-023.

%
%
%
%

\bibliographystyle{amsalpha}

\begin{thebibliography}{A}
\bibitem{AGV} I. Angiono, C. Galindo and L. Vendramin, Hopf braces and Yang-Baxter operators, \emph{Proc. Am. Math. Soc.} {\bf 145}(5) (2017) 1981--1995.
\bibitem{Bach} D. Bachiller, Solutions of the Yang-Baxter equation associated to skew left braces, with applications to racks, \emph{J. Knot Theory Ramifications} {\bf 27}(8) (2018) 1850055, 36 pp.
\bibitem{Bax} R.J. Baxter, Partition function of the eight-vertex lattice model, \emph{Ann. Phys.} {\bf 70}(1)(1972) 193--228.
\bibitem{Childs} L.N. Childs, Bi-skew braces and Hopf Galois structures, \emph{New York J. Math.} {\bf 25} (2019) 574--588.
\bibitem{VGRHac} J.~M. Fern\'andez~Vilaboa, R. Gonz\'alez~Rodr\'iguez and B. Ramos~P\'erez, Categorical isomorphisms for Hopf braces, \textit{Hacet. J. Math. Stat.} {\bf 54}(5) (2025) 1872--1896.
\bibitem{VGRRMod} J.M. Fernández Vilaboa, R. González Rodríguez, B. Ramos Pérez and A.B. Rodríguez Raposo, Modules over invertible 1-cocycles, \textit{Turkish J. Math.} {\bf 48}(2) (2024) 248--266.
\bibitem{VGRRProj} J.M. Fernández Vilaboa, R. González Rodríguez, B. Ramos Pérez and A.B. Rodríguez Raposo, Projections of Hopf braces, \emph{Comm. Algebra} {\bf 53}(7) (2025) 3008--3045. 
\bibitem{RGON} R. González Rodríguez, The fundamental theorem of Hopf modules for Hopf braces, \emph{Linear Multilinear Algebra} {\bf70}(20) (2022) 5146--5156.
\bibitem{RGONRAMOS} R. Gonz\'alez~Rodr\'iguez and B. Ramos~P\'erez, About Hopf braces and crossed products, \textit{S\~ao Paulo J. Math. Sci.} {\bf 20}(1) (2026), Paper No. 11.
\bibitem{GV} L. Guarnieri and L. Vendramin, Skew braces and the Yang-Baxter equation, \emph{Math. Comput.} {\bf 86}(307) (2017) 2519--2534.
\bibitem{LST} Y. Li, Y.-H. Sheng and R. Tang, Post-Hopf algebras, relative Rota-Baxter operators and solutions to the Yang-Baxter equation, \textit{J. Noncommut. Geom.} {\bf 18}(2) (2024) 605--630.
\bibitem{PuraE} M.P. López López and E. Villanueva Novoa, The antipode and the (co)invariants of a finite Hopf (co)quasigroup, \emph{Appl. Categor. Struct. } {\bf 21} (2013) 237--247.
\bibitem{MacLane} S. Mac~Lane, {\it Categories for the working mathematician}, Graduate Texts in Mathematics, Vol. 5, Springer, New York-Berlin, 1971.
\bibitem{Sch} P. Schauenburg, On the braiding on a Hopf algebra in a braided category, {\it New York J. Math.} {\bf 4} (1998), 259--263
\bibitem{Yang} C.N. Yang, Some exact results for the many-body problem in one dimension with repulsive delta-function interaction, \emph{Phys. Rev. Lett.} {\bf 19} (1967) 1312--1315.

\end{thebibliography}

\end{document}